\documentclass[]{amsart}
\usepackage[latin1]{inputenc}
\usepackage[english]{babel}
\usepackage{amsmath}
\usepackage{amsthm}
\usepackage{amssymb}
\usepackage{amsmath, amsthm, amsfonts, mathrsfs, amsfonts, mathtools}
\usepackage{latexsym}
\usepackage{textcomp}
\usepackage{enumerate}
\usepackage{ulem}

\usepackage{xcolor}

\theoremstyle{definition}
\newtheorem{definition}{Definition}[section]

\theoremstyle{remark}
\newtheorem{remark}[definition]{Remark}

\theoremstyle{plain}

\theoremstyle{plain}
\newtheorem{lemma}[definition]{Lemma}

\theoremstyle{plain}
\newtheorem{theorem}[definition]{Theorem}
\newtheorem{proposition}[definition]{Proposition}

\theoremstyle{plain}
\newtheorem*{conj}{Conjecture}

\newtheorem{theoremA}[]{Theorem}

\theoremstyle{plain}
\newtheorem*{theoremB}{Theorem}

\theoremstyle{plain}
\newtheorem{hypothesis}[definition]{Hypothesis}

\newcommand{\C}{\mathrm{C}}

\newcommand{\F}{\mathcal{F}}

\newcommand{\N}{\mathrm{N}}

\newcommand{\X}{\mathcal{X}} 
\newcommand{\orb}{\mathcal{O}}

\newcommand{\Aut}{\mathrm{Aut}}
\newcommand{\Hom}{\mathrm{Hom}}

\newcommand{\Out}{\mathrm{Out}}
\newcommand{\Inn}{\mathrm{Inn}}

\newcommand{\Iso}{\mathrm{Iso}}

\newcommand{\lmod}{{\operatorname{\mathbf{-mod}}}}
\newcommand{\Res}{{\rm Res}}
\newcommand{\Ind}{{\rm Ind}}
\newcommand{\tr}{{\rm tr}} % relative trace
\newcommand{\cj}[2]{\prescript{#1}{}{#2}} % left conjugation
\newcommand{\lqt}{\backslash} % left quotient
\newcommand{\norm}{\mathrel{\unlhd}}

\def \L {\textbf{cc}}

\def \Z {\mathrm {Z}}

\def \iso {{\rm iso}}

\begin{document}

% Title of document, usually lower case except for first word and
% proper nouns.  Avoid unnecessary symbols.

% If the title is too long for the running head, use the following
% command to specify a short title:
%\shorttitle{Shorter title}

\title{Sharpness of saturated fusion systems on a Sylow $p$-subgroup of ${\rm G}_2(p)$}

% First Author           

% First Author's email address.  Must come before \address.

% First author's postal address, with *line breaks* like below.  Do
% not use \\ to separate lines.  It will appear on one line in the
% article, but will be used exactly as typed below to send a
% complimentary copy of the journal, so ensure that it is a complete
% postal address with line breaks as would appear on an envelope.
%
% Also, if it is not obvious, please add a comment with % that says
% what part is a district within a city, what part is a city, 
% what part is a region or province and what part is a postal code.
%
% We recommend using the most current address possible.  
% If you want the journal copy sent to a different address than
% the one below, please let Dan Christensen <jdc@uwo.ca> know.
% Do not end in a period.

% If needed, use a \thanks command, but not inside the author command.
% This should be brief, with longer remarks at the end of the introduction,
% as shown below.
% \thanks{The first author was supported in part by a grant.}

\author[V. Grazian]{Valentina Grazian}
\email{valentina.grazian@unimib.it}
\address{
Department of Mathematics and Applications,
University of Milano -- Bicocca,
Via Roberto Cozzi 55, 20125 Milano,
Italy} 

\author[E. Marmo]{Ettore Marmo}
 \email{ettore.marmo01@universitadipavia.it}
\address{
Department of Mathematics and Applications,
University of Milano -- Bicocca,
Via Roberto Cozzi 55, 20125 Milano, 
Italy}

% If the author names are too long for the running head, use
% the following command to specify a shorter version:
%\shortauthors{DOE, SMITH \andname\ WILLIAMS}

% AMS 2020 Mathematics Subject Classification.  List one or several,
% separated by commas, ending in a period.
% Use \classification[2010]{...} for the 2010 system.
% Make sure your codes are valid with the system you choose:
% https://mathscinet.ams.org/mathscinet/msc/
%\classification{55R35, 55R40, 20J06, 20D20}

% Keywords of the article, usually singular, no leading caps,
% separated by commas, ending with period.
%\keywords{sharpness, homology decomposition, classifying space, fusion systems, Mackey functors.}

% Abstract comes before maketitle.

\begin{abstract}We prove that the D\'iaz-Park's sharpness conjecture holds for saturated fusion systems defined on a Sylow $p$-subgroup of the group ${\rm G}_2(p)$, for $p\geq 5$.
\end{abstract}

% Abstract text, usually no more than 200 words.
% Avoid bibliographic references (\cite) and complicated mathematics.
% Please do not use custom macros here, as this abstract has to be
% able to stand alone with mathjax.  You may use standard tex/latex/AMS macros.

% Please enter the EditFlow identifier for the article here:
%\EditFlow{YYMMDD-Name}

% Leave these items like this, and the journal will fill them in.
%\received{Month Day, Year}   % receive date (for example: October 11, 1999)
%\revised{Month Day, Year}    % date of revision; omit, if no revision;
                             % if multiple revisions, separate by commas
%\accepted{Month Day, Year}  % acceptance date
%\published{Month Day, Year}  % publish date
%\submitted{}      % Name of Journal's Editor, who handled Article 
%\volumeyear{} % Volume Year
%\volumenumber{} % Volume Number 
%\issuenumber{}   % Issue Number
%\startpage{1}     % PageNumber of first page
%\articlenumber{} % Sequence number of article within issue

% If copyright is retained by author, comment this out:
%\owner{International Press}

\maketitle

% Text of Document.  Use constructs such as \section, \subsection,
% \begin{theorem} ... \end{theorem}, \begin{proof} ... \end{proof}, etc.

\section{Introduction}

Given a finite group $G$, a homology decomposition of the classifying space $BG$ for $G$ is a process of glueing together classifying spaces of a suitable collection of subgroups of $G$ in order to obtain a space that is isomorphic to $BG$ in mod-$p$ homology. In \cite{Dwyer} Dwyer proved that the collection of $p$-centric subgroups of $G$ gives rise to a \emph{sharp} homology decomposition, yielding a particularly simple formula for the
cohomology of the classifying space $BG$ in terms of the cohomology of the spaces involved in it:

\begin{theoremB}[Subgroup sharpness for $p$-centric subgroups]
Let $p$ be a prime, let $G$ be a finite group and let $\orb_p^c(G)$ denote the $p$-centric orbit category of $G$. Then $H^j(-; \mathbb{F}_p)$ is acyclic as a controvariant functor over  $\orb_p^c(G)$:
\[{\lim\limits_{\longleftarrow}}_{\orb_p^c(G)}^i H^j(-; \mathbb{F}_p) = 0\]
for all  $i\geq 1, j \geq 0.$
\end{theoremB}

Such a context can be extended  to saturated fusion systems $\F$, thanks to the homotopy theory of fusion systems introduced in \cite{BLO2} and the existence of a unique object  that plays the role of the classifying space $B\F$ for $\F$ (\cite{Ch}). In this view, in \cite{DP} D\'iaz and Park stated the following sharpness conjecture for fusion systems, not only passing from finite groups to saturated fusion systems, but also replacing the cohomology functor $H^j(-; \mathbb{F}_p)$ with a more general Mackey functor $M$ (See Definition \ref{def:mackey}):

\begin{conj}[Sharpness for fusion systems]
Let $p$ be a prime, let $\F$ be a saturated fusion system on a finite $p$-group $S$, let $\orb(\F^c)$ denote the centric orbit category of $\F$ and let $M = (M^*, M_*) : \orb(\F) \to \mathbb{F}_p\lmod$
be a Mackey functor over the full orbit category $\orb(\F)$ of $\F$.
Then the contravariant part $M^*$ over $\orb(\F^c)$ is acyclic:
\[{\lim\limits_{\longleftarrow}}_{\orb(\F^c)}^i M^*|_{\orb(\F^c)} = 0\] for all $i > 0$.
\end{conj}

In particular, the sharpness conjecture for fusion systems implies the acyclicity of the cohomology functor $H^j(-; \mathbb{F}_p) \colon  \orb(\F^c) \to \mathbb{F}_p\lmod$ for any $j \geq 0$, which is one of the open problems listed in \cite[III.7]{AKO}.

As a direct consequence of Dwyer's subgroup sharpness theorem, the conjecture holds for the cohomology functor $H^j(-; \mathbb{F}_p) \colon  \orb(\F) \to \mathbb{F}_p\lmod$ and any fusion system $\F=\F_S(G)$ that is realized by a finite group $G$ containing $S$ as a Sylow $p$-subgroup. In \cite{DP} the authors showed that in fact all fusion systems realized by finite groups satisfy the sharpness conjecture for any Mackey functor $M$. Therefore it remains to test the conjecture on the so called exotic fusion systems, that is, saturated fusion systems not realizable by finite groups. Using the information on saturated fusion systems defined on a $p$-group having a maximal subgroup that is abelian uncovered by Oliver(\cite{p.index}), D\'iaz and Park proved that the sharpness conjecture also holds in this scenario (see \cite[Theorem C]{DP}).
In this paper, we apply D\'iaz and Park's methodology to saturated fusion systems defined on a Sylow $p$-subgroup of the group ${\rm G}_2(p)$, for $p\geq 5$, to prove that the sharpness conjecture holds for this family of saturated fusion systems as well.

\begin{theoremA}\label{th:main}
Let $p\geq 5$ be a prime, let  $S$ be a Sylow $p$-subgroup of the group ${\rm G}_2(p)$, let $\F$ be a saturated fusion system on $S$ and let
$M = (M^*, M_*) : \orb(\F) \to \mathbb{F}_p\lmod$
be a Mackey functor over $\orb(\F)$.
Then
\[{\lim\limits_{\longleftarrow}}_{\orb(\F^c)}^i M^*|_{\orb(\F^c)} = 0\] for all $i > 0$.
\end{theoremA}

Indeed, Theorem \ref{th:main} is a consequence of  the following stronger result involving simple Mackey functors $S_{Q,V}$ (Definition \ref{def:mackey.SQV}):

\begin{theoremA}\label{main}
Let $p\geq 5$ be a prime, let  $S$ be a Sylow $p$-subgroup of the group ${\rm G}_2(p)$ and let $\F$ be a saturated fusion system on $S$. Suppose $Q \leq S$ is not $\F$-centric and $V$ is a simple $\mathbb{F}_p\Out_\F(Q)$-module. Then for all $\F$-centric subgroups $P, R \leq S$ such that their intersection $T = P \cap R$ is not $\F$-centric the composition 
\begin{equation}\label{composition}
S_{Q,V}(P) \xrightarrow{\Res_T^P} S_{Q, V}(T) \xrightarrow{\Ind_T^R} S_{Q, V}(R)	
\end{equation}
is zero.
\end{theoremA}

One of the main ingredients used in the proof of Theorem \ref{main} is the knowledge of the structure of the $\F$-essential subgroups of $S$, first described in \cite{G2p} and then in more generality in \cite{GP}, while characterizing saturated fusion systems defined on a $p$-group having maximal nilpotency class. 

\vspace{0.3cm}
We remark that in a recent paper \cite{punctured}, Henke, Libman and Lynd proved that the families of fusion systems affording the structure of punctured groups, such as the Benson-Solomon system $\F_{Sol}(3)$ \cite[III.6.3, Theorem 6.7]{AKO}, the Parker-Stroth systems \cite{ParkerStroth} and the Clelland-Parker systems \cite{ClellandParker2010} in which all essential subgroups are abelian, satisfy the sharpness conjecture (although they gave an explicit proof only for the cohomology functor and not for a generic Mackey functor).
The sharpness conjecture for fusion systems is still open for other families of exotic fusion systems. 

\vspace{0.5cm}
Organization of the paper. In Section 2 we recall the definitions of orbit category, Mackey functors and simple Mackey functors for fusion systems and we prove some basic properties. In Section 3 we work under the assumption that $\F$ is a saturated fusion system on a finite $p$-group $S$ and $Q, P$ and $R$ are subgroups of $S$ for which the composition of maps (\ref{composition}) is non-zero. We underline that the results proved in this section hold for any finite $p$-group $S$ and can represent the starting point for the proof of the sharpness conjecture for other families of saturated fusion systems. In Section 4 we specialize our investigation on $p$-groups of maximal nilpotency class, given that the Sylow $p$-subgroups of the group $\rm{G}_2(p)$ for $p\geq 5$ are of this type.  Finally, Section 5 contains the proof of Theorems \ref{th:main} and \ref{main}. 

\section{Preliminaries}

We refer to \cite{AKO} for an introduction to saturated fusion systems.
Here we recall the definitions of orbit categories and Mackey functors adapted to the setting of fusion systems, following the approach of \cite{DP}.

\begin{definition}\label{def:orbit.category} Let $\F$ be a fusion system on a $p$-group $S$. The \emph{orbit category} of $\F$ is a category $\orb(\F)$ whose objects are the subgroups of $S$ and whose morphisms are given by the left quotient:
\[
    \Hom_{\orb(\F)}(P, Q) = \Inn(Q) \lqt \Hom_\F(P, Q).
\] 
If $\varphi \in \Hom_\F(P,Q)$ for some $P,Q\leq S$, we denote by $[\varphi]$ the corresponding morphism in $\Hom_{\orb(\F)}(P, Q)$.

%The composition is defined as we would expect:
%\[
%    [\psi][\varphi] := [\psi \circ \varphi]
%\]
%for $\varphi \in \Hom_\F(P, Q)$ and $\psi \in \Hom_\F(Q, R)$.
\end{definition}

\begin{remark}
    Any morphism $\alpha : Q' \to P$ in the orbit category is injective and can be described as a composition of the equivalence class of an inclusion with an isomorphism, i.e. \[
    \alpha = [\iota_Q^P][\varphi]
    \]
    where $\iota_Q^P$ is the inclusion map $Q \hookrightarrow P$ of a subgroup $Q$ of $P$ and $\varphi : Q' \to Q$ is an $\F$-isomorphism. 
\end{remark}

If $\X$ is an $\F$-overconjugacy closed collection of subgroups of $S$ (that is, if $P \in \X$ then all overgroups of $P$ in $S$ and all $\F$-conjugates of $P$ are contained in $\X$ as well), then  we have an analogous definition for $\orb(\X)$ in which the objects are exactly the subgroups in $\X$. An important example is given by the \emph{centric-orbit category} $\orb(\F^c)$ where $\F^c$ is the collection of $\F$-centric subgroups of $S$, that is, all subgroups $P$ of $S$ such that $C_S(Q)\leq Q$ for every subgroup $Q$ of $S$ that is $\F$-conjugate to $P$.

We will now recall the definition of \emph{$\X$-restricted Mackey functor for $\F$} as introduced in \cite[Defintion 2.1]{DP}, specializing it for the finite field $\mathbb{F}_p$.

\begin{definition}\label{def:mackey}  Let $\F$ be a saturated fusion system on a $p$-group $S$ and let $\X$ be an $\F$-overconjugacy closed collection of subgroups of $S$. A \emph{$\X$-restricted Mackey functor for $\F$} is given by a pair of functors
\[
    M = (M^*, M_*) : \orb(\X) \to \mathbb{F}_p\lmod
\]
satisfying the following conditions:
\begin{enumerate}
    \item{\rm (Bivariance)} The functors $M^*$ and $M_*$ are respectively covariant and contravariant and take the same values on all objects of $\orb(\X)$. For all $P \leq S$ we denote $M^*(P) = M_*(P)$ by $M(P)$.
    \item{\rm (Isomorphism)} For all isomorphisms $\alpha$ in $\orb(\X)$ we have $M^*(\alpha) = M_*(\alpha^{-1})$.
    \item{\rm ($\X$-truncated Mackey decomposition)} If $Q, R \leq P \leq S$ with $Q, R, P \in \X$ we have 
    \[
			M^*([\iota_R^P]) \circ M_*([\iota_Q^P]) = \sum\limits_{x \in [R\backslash P/Q]_\X} M_*([\iota_{R\cap \cj xQ}^R]) \circ M^*([\iota_{R\cap \cj xQ}^{\cj xQ}]) \circ M_*([c_x|_Q]).		
    \]
    where $[R\backslash P/Q]_\X$ denotes a set of representatives of double cosets of $P$ of the form $RxQ$ with $Q \cap \cj xR \in \X$
\end{enumerate}

When $\X$ consists of all subgroups of $S$, then $\orb(\X) = \orb(\F)$ and we simply say that $M$ is a Mackey functor for $\F$.
\end{definition}

If the functor $M$ is clear from the context we will adopt the following notation: for subgroups $Q \leq P \leq S$ and an $\F$-isomorphism $\varphi$ let
\begin{itemize}
    \item $\Ind_Q^P := M_*([\iota_Q^P]) : M(Q) \to M(P)$;
    \item $\Res_Q^P := M^*([\iota_Q^P]) : M(P) \to M(Q)$;
    \item $\Iso(\varphi) := M_*([\varphi])$.
\end{itemize}

\begin{remark}
    For any Mackey functor $M : \orb(\F) \to \mathbb{F}_p\lmod$, the image $M(P)$ of $P \leq S$ always admits a structure of $\mathbb{F}_p\Out_\F(P)$-module. 
\end{remark}

We will be interested in restricting the domain of a Mackey functor $M$ to the centric-orbit category $\orb(\F^c)$. It is however not always the case that such a restriction will satisfy the $\F^c$-truncated Mackey formula. \\
In other words, using the notation of Definition \ref{def:mackey}, for a Mackey functor 
\[
    M = (M^*, M_*) : \orb(\F) \to \mathbb{F}_p\lmod
\] 
the restriction  $M|_{\orb(\F^c)} = (M^*|_{\orb(\F^c)}, M_*|_{\orb(F^c)})$ is not necessarily an $\F^c$-restricted Mackey functor. A sufficient condition is given by the following criterion.

\begin{proposition}\cite[Proposition 4.4]{DP}\label{prop:trunc}
    Let $\F$ be a saturated fusion system on $S$ and $M$ a Mackey functor for $\F$ over a field of characteristic $p$. Then the restriction $M|_{\orb(\F^c)}$ is an $\F^c$-truncated Mackey functor for $\F$ if the composite
    \[
        M(P) \xrightarrow{\Res_{P\cap R}^P} M(P\cap R) \xrightarrow{\Ind_{P\cap R}^R} M(R)
    \]
    is zero whenever $P, R \leq S$ are $\F$-centric and $P\cap R$ is not $\F$-centric.
\end{proposition}

We are now going to introduce the important  family $S_{Q,V}$ of Mackey functors, parametrized by a subgroup $Q$ of $S$ and a left $\mathbb{F}_p\Out_\F(Q)$-module $V$.

\begin{definition}
    Let $\F$ be a fusion system on a $p$-group $S$ and let $P, R \leq S$ be such that  $\Iso_\F(P,R) \neq \emptyset$. If $\alpha \in \Iso_\F(P,R)$ and $V$ is a $\mathbb{F}_p\Out_\F(P)$-module, then we denote by $\cj \alpha V$ the structure of $\mathbb{F}_p\Out_\F(R)$-module on $V$ induced by the following action
    \[
        [\gamma] \cdot v := [\alpha^{-1}\gamma\alpha]v 
    \]
    for all $[\gamma] \in \Out_\F(R)$ and $v \in V$.
\end{definition}

\begin{definition}\label{def:mackey.SQV}
    Let $\F$ be a saturated fusion system on a $p$-group $S$. Consider a subgroup $Q \leq S$ and a left $\mathbb{F}_p\Out_\F(Q)$-module $V$. We are going to define a Mackey functor $S_{Q,V} : \orb(\F) \to \mathbb{F}_p\lmod$.

\begin{enumerate}
    \item[(a)] For any subgroup $L \leq S$ which is $\F$-conjugate to $Q$, fix an $\F$-isomorphism $\alpha : Q \to L$ and define $S_{Q, V}(L) := \cj \alpha V$. \medskip
    \item[(b)] For $P \leq S$ we define
    \[
        S_{Q, V}(P) := \bigoplus_{\alpha : Q \cong L \leq_P P} \tr_L^{N_P(L)}(\cj \alpha V), 
    \]
    where the sum is indexed over the $P$-conjugacy classes of subgroups $L$ that are $\F$-isomorphic to $Q$. We will sometimes denote the direct summands of $S_{Q, V}(P)$ by $V_{\alpha, L}(P)$. 

    \item[(c)] Consider an isomorphism $[\varphi] : L_1 \to L_2$ in $\orb(\F)$. Suppose we fixed $\F$-isomorphisms $\alpha_i : Q \to L_i$ for $i = 1, 2$. Then we define
    \begin{align*}
        \Iso([\varphi]) := {S_{Q, V}}_*([\varphi]) : \cj{\alpha_1}{V} &\to \cj{\alpha_2}{V} \\
        v &\mapsto [\alpha_2^{-1}\varphi\alpha_1] \cdot v,
    \end{align*}
   Finally we define ${S_{Q, V}}^*([\varphi]) := \Iso([\varphi^{-1}])$. \medskip
    
    \item[(d)] Consider an isomorphism $[\varphi] : P \to R$ in $\orb(\F)$. We define $\iso([\varphi]) := {S_{Q, V}}_*([\varphi])$ on each summand $V_{\alpha, L}(P)$ using point (c):
        \[
            \Iso([\varphi])|_{V_{\alpha, L}(P)} := \Iso([\varphi|_L]).
        \]
    Finally we define ${S_{Q, V}}^*([\varphi]):= \Iso([\varphi^{-1}])$. \medskip
    
    \item[(e)] Consider an inclusion of subgroups $\iota_P^R : P \hookrightarrow R$ of $S$ and the corresponding morphism $[\iota_P^R]$ in $\orb(\F)$. We define $\Ind_P^R := {S_{Q, V}}_*([\iota_P^R]) : S_{Q, V}(P) \to S_{Q, V}(R)$ on each direct summand as follows:
    \begin{align*}
        \Ind_P^R|_{V_{\alpha, L}(P)} : V_{\alpha, L}(P) &\to V_{\alpha, L}(R) \\
        v &\mapsto \tr_{N_P(L)}^{N_R(L)}(v).
    \end{align*}
    \medskip
    \item[(f)] The contravariant image, $\Res_P^R := {S_{Q,V}}^*([\iota_P^R]) : S_{Q, V}(R) \to S_{Q, V}(P)$, is defined as follows. Consider a direct summand $V_{\alpha, L}(R)$, if for all $r \in R$ the conjugate $\cj r L \nleq P$ then $V_{\alpha, L}(R)$ is mapped to zero. Suppose otherwise that there are $R$-conjugates of $L$ in $P$, fix the representatives for their $P$-conjugacy classes $L_i = \cj{r_i}{L}$ and the $\F$-isomorphisms $\beta_i = c_{r_i}\circ \alpha$. Then $V_{\alpha, L}(R)$ is mapped to $V_{\beta_1, L_1}(P) \oplus \cdots \oplus V_{\beta_k, L_k}(P)$ via the diagonal map:
    \[
        v \mapsto v + \cdots + v.
    \]
\end{enumerate}
\end{definition}

\begin{remark}Given two Mackey functors $M, N : \orb(\F) \to \mathbb{F}_p\lmod$ we say that $N$ is a subfunctor of $M$ if for all $P \in \orb(\F)$ we have $N(P) \leq M(P)$. Accordingly, a Mackey functor is \emph{simple} if it admits no subfunctor than itself and the zero functor. The work of \cite{Webb}, generalized to the context of fusion systems by D\'iaz and Park \cite{DP}, shows that $S_{Q,V}$ (as defined in Definition \ref{def:mackey.SQV}) is in fact a simple Mackey functor as soon as $V$ is a simple $\mathbb{F}_p\Out_\F(Q)$-module and that all simple Mackey functors are of this form.
\end{remark}

We now state and prove a property first uncovered by Webb \cite[Proposition 2.8]{Webb}, in the language used in \cite{DP}.

\begin{lemma}\label{lem:trace.nonzero} Let $\F$ be a saturated fusion system on a $p$-group $S$ and let $S_{Q,V} : \orb(\F) \to \mathbb{F}_p\lmod$ be the simple Mackey functor associated to the subgroup $Q\leq S$ and the left $\mathbb{F}_p\Out_\F(Q)$-module $V$.
    Let $L \leq P \leq S$ be such that $L \in Q^\F$.
    If $\alpha\in \Iso_\F(Q,L)$ and $V_{\alpha, L}(P) = \tr_L^{N_P(L)}(\cj \alpha V) \neq 0$, then  $C_P(L) \leq L$.
\end{lemma}
\begin{proof}
    Suppose $C_P(L) \nleq L$, then $L < LC_P(L) \leq N_P(L)$ and \[
        V_{\alpha, L}(P) = \tr_L^{N_P(L)}(\cj \alpha V) = \tr_{LC_P(L)}^{N_P(L)} \tr_L^{LC_P(L)}(\cj \alpha V).
    \]
    The action of $N_P(L)$ on the $\mathbb{F}_p\Out_\F(L)$-module $\cj \alpha V$ is given by the composition of maps $N_P(L) \to \Aut_P(L) \subseteq \Aut_\F(L) \to \Out_\F(L)$. Both $L$ and $C_P(L)$ get mapped to the trivial subgroup under this composition, hence their action on $\cj \alpha V$ is trivial and so $\tr_L^{LC_P(L)}(\cj \alpha V) = 0$ which implies $V_{\alpha, L}(P) = 0$, a contradiction.
\end{proof}

We now set a notation for the subgroups of $P$ that are $\F$-conjugate to $Q$ and self-centralizing in $P$, whose importance is highlighted by Lemma \ref{lem:trace.nonzero}.

\begin{definition}  Let $\F$ be a saturated fusion system on a $p$-group $S$. For $P,Q \leq S$, we define the set
    \[
        \L(P,Q) = \{ L \leq P \ |\ \Iso_\F(Q, L) \neq \emptyset {\rm \ and\ } C_P(L) \leq L \}.
    \]

\end{definition}

\begin{lemma}\label{lem:SQV.nonzero}
Let $\F$ be a saturated fusion system on a $p$-group $S$ and let $S_{Q,V} : \orb(\F) \to \mathbb{F}_p\lmod$ be the simple Mackey functor associated to the subgroup $Q\leq S$ and the left $\mathbb{F}_p\Out_\F(Q)$-module $V$.
    Let $P \leq S$.
    \begin{enumerate}[(i)]
        \item If $S_{Q, V}(P) \neq 0$, then $\L(P,Q) \neq \emptyset$;
        \item If $L \in \L(P,Q)$ and $x \in P$ then $\cj x L \in \L(P, Q)$.
        \item If $L \in \L(P,Q)$ and $L\leq T \leq P$ then $L \in \L(T, Q)$.
   
    \end{enumerate}
\end{lemma}
\begin{proof}
    Suppose $S_{Q, V}(P) \neq 0$. By Definition \ref{def:mackey.SQV}(a)(b) there exists a subgroup $L$ of $P$ such that $L \in Q^\F$ and $V_{\alpha,L}(P) \neq \emptyset$. Part (i) now follows from Lemma \ref{lem:trace.nonzero}.
    
    Consider now $L \in \L(P, Q)$ and $\alpha \in \Iso_\F(Q,L)$. For all $x \in P$ the composition $c_x \circ \alpha : Q \to \cj x L$ is still an isomorphism in $\F$. Also, since $C_P(L) \leq L$ we have $C_P(\cj x L) = \cj{x}{C_P(L)} \leq \cj x L$ proving $\cj x L \in \L(P, Q)$. Thus part (ii) holds.
    
    Finally, if $L \in \L(P, Q)$ and $L \leq T \leq P$, then $C_T(L) \leq C_P(L) \leq L$ which implies $L \in \L(T, Q)$ and proves part (iii).

\end{proof}

\newpage
\section{First results}
In this section we work under the following:

\begin{hypothesis}\label{hp:QV}
Suppose $p$ is an odd prime, $S$ is a finite $p$-group and $\F$ is a saturated fusion system on $S$. Fix a subgroup $Q \leq S$ that is not $\F$-centric and a simple $\mathbb{F}_p\Out_\F(Q)$-module $V$.
\end{hypothesis}

\begin{lemma}\label{lem:K-abelian}
Suppose Hypothesis \ref{hp:QV} holds. If $K\leq S$ is abelian and $S_{Q,V}(K)\neq 0$, then $K$ is $\F$-conjugate to $Q$. In particular if $\alpha \in \Iso_\F(Q,K)$ then $S_{Q,V}(K) = \cj{\alpha}{V}$.
\end{lemma}

\begin{proof} By Lemma \ref{lem:SQV.nonzero} the assumption $S_{Q,V}(K) \neq 0$ implies that there exists a subgroup $L \leq K$ which is $\F$-conjugate to $Q$ and such that $C_K(L) \leq L$. Since $K$ is abelian, we obtain $K = C_K(L) \leq L$ and so $K = L$ is $\F$-conjugate to $Q$. \end{proof}

From now on we consider the case of subgroups $Q,P$ and $R$ of $S$ for which the composition of maps (\ref{composition}) is non-zero, that is, we assume 

\begin{hypothesis}\label{hp:contra}
Suppose Hypothesis \ref{hp:QV} holds and assume that $P$ and $R$ are $\F$-centric subgroups of $S$ such that $T = P\cap R$ is not $\F$-centric and
the composition 
\begin{equation}\tag{1}
S_{Q,V}(P) \xrightarrow{\Res_T^P} S_{Q, V}(T) \xrightarrow{\Ind_T^R} S_{Q, V}(R)	\end{equation}
is non-zero.
\end{hypothesis}

It is clear that under the assumptions of Hypothesis \ref{hp:contra} the group $T$ is properly contained in both $P$ and $R$; in particular $P$ and $R$ are proper subgroups of $S$.

\begin{lemma}\label{lem:basic.prop}
Suppose Hypothesis \ref{hp:contra} holds. Then 
\begin{enumerate}[(i)]
\item $S_{Q, V}(P)$, $S_{Q, V}(R)$ and $S_{Q, V}(T)$ are all non-zero;
\item $T < P < S$ and $T < R < S$;
\item $Z(S) \leq T$; and
\item $P$ and $R$ are not abelian.
\end{enumerate}
\end{lemma}
\begin{proof}
Part (i) follows from the fact that the composition (\ref{composition})    is non-zero by assumption. Part (ii) is a direct consequence of the fact that $P$ and $R$ are $\F$-centric while $T$ is not $\F$-centric. Since $P$ and $R$ are $\F$-centric, we also get $Z(S) \leq C_S(P) \cap \C_S(R) \leq P \cap R = T$, proving (iii). Finally, aiming for a contradiction, suppose $P$ is abelian. Then $P$ is $\F$-conjugate to $Q$ by Lemma \ref{lem:K-abelian}, contradicting the fact that $P$ is $\F$-centric while $Q$ is not $\F$-centric. The same reasoning works for $R$. Therefore $P$ and $R$ are not abelian and this is (iv). \end{proof}

\begin{lemma}\label{lem:T.contains.centralizers}
    Suppose Hypothesis \ref{hp:contra} holds. Then the following holds:
    \begin{enumerate}[(i)]
        \item $C_S(P) \leq T$;
        \item $C_S(R) \leq T$.
    \end{enumerate}
\end{lemma}
\begin{proof}
    By Lemma \ref{lem:basic.prop}(i) we get that $S_{Q, V}(P)$ and $S_{Q, V}(T)$ are nonzero. Hence Lemma \ref{lem:SQV.nonzero} (i) implies that both $\L(P, Q)$ and $\L(T, Q)$ are nonempty. 
    Aiming for a contradiction, suppose that for all $L \in \L(P, Q)$ we have $L \nleq T$. By Lemma \ref{lem:trace.nonzero} and Definition \ref{def:mackey.SQV}(f) we conclude that $\Res_T^P$ is the zero map, contradicting  Hypothesis \ref{hp:contra}. So there exists $L \in \L(P,Q)$ such that $ L \leq T$. Since $P$ is $\F$-centric, we get 
     \[C_S(P) = Z(P) \leq C_P(L) \leq L \leq T.\]
     So $C_S(P) \leq T$. \\
    To prove part (ii) note that by Lemma \ref{lem:SQV.nonzero} a necessary condition for $\Ind_T^R$ to be nonzero is that there must exist $L \in \L(T,Q)$ such that $L \in \L(R,Q)$. For such subgroup $L$ we then have $\Z(R) \leq C_R(L) \leq L \leq T$ and so $Z(R) \leq T$. But $R$ is $\F$-centric, so we conclude $C_S(R) = Z(R) \leq T$ as desired.
\end{proof}

\begin{lemma}\label{lem:T.abelian}
Suppose Hypothesis \ref{hp:contra} holds and $T$ is abelian. Then
\begin{enumerate}[(i)]
\item $T$ is $\F$-conjugate to $Q$, so  there exists $\alpha \in \Iso_\F(Q,T)$ such that $S_{Q, V}(T) = \cj{\alpha}{V}$; and
\item if  $T < C_K(T) \leq K \leq S$ then $\tr_T^K(S_{Q, V}(T)) = 0$.
\end{enumerate}
\end{lemma}

\begin{proof} Since $T$ is abelian and $S_{Q, V}(T) \neq 0$ by Lemma \ref{lem:basic.prop}, part (i) follows from Lemma \ref{lem:K-abelian}. In particular there exists $\alpha \in \Iso_\F(Q,T)$ and $S_{Q, V}(T) = \cj{\alpha}{V}$. 

If $T < C_K(T) \leq K$, then \[\tr_T^K(\cj \alpha V) = \tr_{C_K(T)}^K \tr_T^{C_K(T)}(\cj \alpha V).\]
Since the action of $C_K(T)$ on $\cj \alpha V$ (a $\mathbb{F}_p\Out_\F(T)$-module) is trivial, by definition of relative trace we obtain 
\[\tr_T^{C_K(T)}(v) = [C_K(T) : T]v,\] 
for all $v \in \cj{\alpha}{V}$. As $C_K(T)$ and $T$ are $p$-groups and we are working in characteristic $p$, the map $\tr_T^{C_K(T)}$ is identically zero. Hence $\tr_T^K(\cj \alpha V) = 0$, proving (ii).
\end{proof}

\begin{lemma}\label{lem:T.abelian.normal}
Suppose Hypothesis \ref{hp:contra} holds and $T$ is abelian. 
\begin{enumerate}[(i)]
\item If $T$ is normal in $U \in \{P,R\}$ then $T = C_U(T)$.
\item If $T$ is normal in both $P$ and $R$ then $\Out_P(T) \neq \Out_R(T)$. 
\end{enumerate}
\end{lemma}

\begin{proof}
As $T$ is abelian, Lemma \ref{lem:T.abelian} (i) tells us that there exists an $\F$-isomorphism $\alpha : Q \cong T$ and that $S_{Q,V}(T) = \cj \alpha V$. Let $U\in \{P,R\}$.  As $T \norm U$, we have \[V_{\alpha, T}(U) = \tr_T^U(\cj \alpha V).\]
Applying the definition of simple Mackey functors we see that if $U=P$ then the restriction $\Res_T^P$ is defined as the zero map on those summands $V_{\beta, L}(P)$ with $L \neq T$ and as the inclusion $\tr_T^P(\cj \alpha V) \hookrightarrow \cj \alpha V$ for $L = T$. So the first part of the composition of maps (\ref{composition}) reduces to $\tr_T^P(\cj \alpha V) \hookrightarrow \cj \alpha V$, that is therefore non-zero by assumption. Similarly, if $U=R$ the induction $\Ind_T^R$ coincides with the relative trace map $\cj \alpha V \to \tr_T^R(\cj \alpha V)$, which is again non-zero.

Now, if $T < C_U(T)$ for $U \in \{P, R\}$ then by Lemma \ref{lem:T.abelian} (ii) we have $\tr_T^U(\cj \alpha V) = 0$, contradicting what we just proved. Hence (i) holds.

To prove (ii), suppose $\Out_P(T) = \Out_R(T)$. Note that, as $T$ is abelian and normal in both $P$ and $R$, we have $\Out_P(T) \cong P/T$ and $\Out_R(T) \cong R/T$. Moreover, by what we showed above, the non-zero composition of maps (\ref{composition}) is reduced to 

\[\tr_T^P(\cj \alpha V) \hookrightarrow \cj \alpha V \to \tr_T^R(\cj \alpha V)\]

Now the relative trace of an element $v \in \cj \alpha V$ with respect to the action of $R$ is
\begin{align*}
    \tr_T^R(v) &= \sum_{x \in [R/T]} [c_x|_T] \cdot v \\
    &= \sum_{c_x|_T \in \Out_R(T)} [c_x|_T] \cdot v \\
    &= \sum_{c_y|_T \in \Out_P(T)} [c_y|_T] \cdot v \\
    &= \sum_{y \in [P/T]} [c_y|_T] \cdot v \\
    &= \tr_T^P(v).
\end{align*}
Hence the maps $v \mapsto \tr_T^R(\tr_T^P(v))$ and $v \mapsto \tr_T^P(\tr_T^P(v))$ coincide. This second map however is the zero map as elements of $\tr_T^P(\cj \alpha V)$ are invariant under the action of $\Out_T(P)$, hence 
\[
    \tr_T^P(\tr_T^P(v)) = \sum_{x \in [P/T]} [c_x|_T]\cdot \tr_T^P(v) = \sum_{x \in [P/T]}\tr_T^P(v) = [P : T]\tr_T^P(v) = 0     
\]  
where the last equality is justified by the fact that $[P:T]$ is a power of $p$ and $\cj \alpha V$ is an $\mathbb{F}_p$-vector space. This proves $v \mapsto \tr_T^R(\tr_T^P(v))$ is the zero map thus composition (\ref{composition}) is also zero, a contradiction.  
\end{proof}

\begin{lemma}\label{lem:T.not.center}
Suppose Hypothesis \ref{hp:contra} holds. Then $Z(S) < T$ and $T$ has order at least $p^2$.
\end{lemma}
\begin{proof}
Note that $Z(S)$ is an abelian normal subgroup of $S$ centralized by $S$. Hence $Z(S) \neq T$ by Lemma \ref{lem:T.abelian.normal}(i) and we conclude using Lemma \ref{lem:basic.prop}(iii).
\end{proof}

\section{The maximal class case}
For $p\geq 5$, the Sylow $p$-subgroups of ${\rm G}_2(p)$ are $p$-groups of maximal nilpotency class. For this reason, we now focus our attention on this class of $p$-groups.

Suppose $S$ is a $p$-group of order $p^n$ having maximal nilpotency class. Recall that the terms of the lower central series for $S$ are defined as
$$\gamma_2(S) := [S,S]; \quad \gamma_i(S) := [\gamma_{i-1},S] \quad \text{ for all } 3 \leq i \leq n.$$
Also, the terms of the upper central series of $S$ are
$$ \Z_1(S) = \Z(S); \quad \Z_i(S)/\Z_{i-1}(S) = \Z( S/Z_{i-1}(S) ) \quad \text{ for all } 2 \leq i \leq n-1.$$
If $|S|\geq p^4$, we set
$$ \gamma_1(S) := \C_S(\gamma_2(S)/\gamma_4(S))$$
and, following the terminology introduced in \cite{GP}, we say that $S$ is exceptional if $\gamma_1(S) \neq \C_S(\Z_2(S))$. Note that both $\gamma_1(S)$ and  $\C_S(\Z_2(S))$ are maximal subgroups of $S$. We refer to \cite{GP} for a detailed analysis of the properties of these groups.

If $p\geq 5$ and $S$ is a Sylow $p$-subgroups of ${\rm G}_2(p)$, then $|S|=p^6$, $\gamma_1(S)$ is extraspecial and $\Z(\gamma_1(S))=\Z(S)$. In particular $S$ is exceptional.

In this section, we assume the following:

\begin{hypothesis}\label{hp:contra.max}
Suppose Hypothesis \ref{hp:contra} holds and $S$ is a $p$-group of maximal nilpotency class.
\end{hypothesis}

A crucial role is played by the characterization of the essential subgroups of saturated fusion systems on maximal calss $p$-groups:

\begin{theorem}\cite[Theorem D]{GP}\label{thm:essentials}
Suppose that $p$ is a prime, $S$ is a $p$-group of maximal class and order at least $p^4$ and $\F$ is a saturated fusion system on $S$. If $E$ is an $\F$-essential subgroup, then either $|E|\leq p^3$, or $E = \gamma_1(S)$ or $E = \C_S(\Z_2(S))$. Furthermore,
if $S$ is exceptional, then $|E|\neq p^3$.
\end{theorem}

Thanks to the work of D\'iaz and Park \cite{DP}, we can exclude the case in which the group $\gamma_1(S)$ is abelian:

\begin{lemma}\label{gamma1.not.ab}
Suppose Hypothesis \ref{hp:contra.max} holds. Then $\gamma_1(S)$ is not abelian.
\end{lemma}

\begin{proof}
The proof of \cite[Theorem C]{DP} implies that the composition
\begin{equation*}
S_{Q,V}(P) \xrightarrow{\Res_T^P} S_{Q, V}(T) \xrightarrow{\Ind_T^R} S_{Q, V}(R)	\end{equation*}
is zero, contradicting our assumptions.
\end{proof}

\begin{lemma}\label{lem:Tnorm}
Suppose Hypothesis \ref{hp:contra.max} holds. If $T\norm S$ then 
\begin{enumerate}[(i)]
\item $T= \gamma_i(S)$ for some $i\geq 2$; 
\item $\gamma_{i-1}(S) \leq C_S(T)$ is abelian (so in particular $T$ is abelian); and
\item $C_S(T)$ is not a maximal subgroup of $S$.
\end{enumerate}
In particular $T\neq Z_2(S)$ and $T\neq \gamma_2(S)$.
\end{lemma}

\begin{proof}
First note that the assumption $T\norm S$ implies that the group $T$ is not $S$-centric. Indeed $T$ is fully centralized in $\F$ \cite[Theorem 5.2]{RS} and if it were $S$-centric then $C_S(T) = Z(T)$ and, for any $\F$-conjugate $T'$ of $T$, we would have $|C_S(T)| \geq |C_S(T')| \geq |Z(T')| = |Z(T)| = |C_S(T)|$, and so  $C_S(T') = Z(T')\leq T'$, implying that $T$ is $\F$-centric, a contradiction. 

Since $S$ has maximal nilpotency class and $T$ is normal in $S$ but not a maximal subgroup of $S$ by Lemma \ref{lem:basic.prop} (ii), we deduce that $T= \gamma_i(S)$ for some $i\geq 2$, and this is statement (i). Note that $C_S(\gamma_i(S)) = \gamma_j(S)$ for some $j\geq 1$.
 Since $T$ is not $S$-centric, we get $\gamma_{i-1}(S) \leq C_S(T)$ and so $\gamma_{i-1}(S)$ is abelian, proving (ii).
Now, aiming for a contradiction, suppose $C_S(T)$ is a maximal subgroup of $S$. Thus $|\Aut_S(T)| = [S \colon C_S(T)]  = p$. Since $T$ is abelian and normal in $S$, by Lemma \ref{lem:T.abelian.normal} (i) we get $T=C_P(T)=C_R(T)$. Thus $\Aut_P(T) \neq 1 \neq \Aut_R(T)$, giving
\[\Aut_P(T) = \Aut_S(T) = \Aut_R(T).\]
This implies $\Out_P(T) = \Out_R(T)$ contradicting Lemma \ref{lem:T.abelian.normal} (ii). Hence  (iii) holds.

For the chaser of the Lemma, note that $C_S(Z_2(S))$ is a maximal subgroup of $S$ (contrary to the statement of part (iii)) and if $T=\gamma_2(S)$ then by part (ii) the group $\gamma_1(S)$ should be abelian, contradicting Lemma \ref{gamma1.not.ab}. 
\end{proof}

\begin{lemma}\label{lem:T-in-W}
Suppose Hypothesis \ref{hp:contra.max} holds. Then either $T \leq \gamma_1(S)$ or $T \leq C_S(Z_2(S))$.
\end{lemma}

\begin{proof} Aiming for a contradiction, suppose $T \nleq \gamma_1(S)$ and $T \nleq C_S(Z_2(S))$. By Lemma \ref{lem:T.not.center}
 we have $Z(S) < T$ and $|T|\geq p^2$. So $P$ and $R$ are subgroups of  $S$ of order at least $p^3$ not contained in $\gamma_1(S)$. By \cite[Lemma 2.12]{pearls} we deduce that $Z_2(S) \leq P \cap R = T$. Note that $Z_2(S) \neq T$ by Lemma \ref{lem:Tnorm}, so $|T| \geq p^3$. In particular, Theorem \ref{thm:essentials} and the fact that $T$ is not $\F$-centric imply that $T$ is not contained in any $\F$-essential subgroup of $S$ and so, by the Alperin-Goldschmidt's fusion theorem (\cite[Theorem I.3.5]{AKO}), all $\F$-conjugates of $T$ can be found applying automorphisms of $\Aut_\F(S)$.

Let $t \in T \setminus (\gamma_1(S) \cup \C_S(Z_2(S)))$. By \cite[Lemma 3.4]{GP} we have $C_S(t) = \langle Z(S), t\rangle$ and so by Lemma \ref{lem:basic.prop} (iii) we get
\[C_S(T) \leq C_S(t) = \langle Z(S), t\rangle \leq T.\] 
Thus $T$ is $S$-centric. Since $Z(S)$, $\gamma_1(S)$ and $C_S(Z_2(S))$ are characteristic subgroups of $S$ we deduce that all $\F$-conjugates of $T$ contain $Z(S)$ and are not contained in $\gamma_1(S)$ or $C_S(Z_2(S))$. The same reasoning used for $T$ then shows that all its $\F$-conjugates are $S$-centric, impying that $T$ is $\F$-centric, a contradiction. This proves the statement.
\end{proof}

\begin{lemma}\label{lem:T.not.S.centric}
   Suppose Hypothesis \ref{hp:contra.max} holds and $\gamma_1(S)$ is extraspecial. Then the following holds
\begin{enumerate}
\item either $T \norm S$ or $T\leq \gamma_1(S)=\N_S(T)$; and 
\item $T$ is not $S$-centric.
\end{enumerate}
\end{lemma}
\begin{proof} 
By Lemmas \ref{lem:T.not.center} and \ref{lem:T-in-W}
we have $Z(S) < T \subseteq \gamma_1(S) \cup C_S(Z_2(S))$. Thus \cite[Lemma 3.7]{GP} implies that either $\N_S(T) \leq \gamma_1(S)$ or $T \norm S$. If $T$ is not normal in $S$, then  $T \leq N_S(T) \leq \gamma_1(S)$ and since $\gamma_1(S)$ is extraspecial and $[\gamma_1(S),\gamma_1(S)] = Z(S) \leq T$ we deduce that $\gamma_1(S) = \N_S(T)$.

As for the second statement, aiming for a contradiction, suppose $T$ is $S$-centric, that is, $C_S(T) \leq T$. As $\F$ is saturated, if $T$ is fully normalized in $\F$ then it is $\F$-centric, contradicting Hypothesis \ref{hp:contra.max}. 
    Hence $T$ is $\F$-conjugate to a subgroup $T'$ of $S$ with $|N_S(T')| > |N_S(T)|$. Using Lemma \ref{lem:T.not.center} and the fact that $\gamma_1(S)$ is extraspecial, we get $[\gamma_1(S), \gamma_1(S)] = Z(S) < T$ and so $T \norm \gamma_1(S)$. Thus $N_S(T) = \gamma_1(S)$ and $T' \norm S$.
      Also, $|T|=|T'|$, so $T'$ is not a maximal subgroup of $S$ and we deduce that $T' \leq \gamma_1(S)$. As $\gamma_1(S)$ is extraspecial and $T$ and $T'$ are normal subgroups of $\gamma_1(S)$ having the same order, we deduce that  
      \[|C_S(T')| = |C_S(T)|.\]
Therefore       
      \[|C_S(T')| = |C_S(T)| \leq |T| = |T'|.\] 
      Now, since $T'$ is normal in $S$, which has maximal nilpotency class, we have $T'= \gamma_i(S)$ and $C_S(T') = \gamma_j(S)$ for some $i,j \geq 1$. Thus $C_S(T') \leq T'$. Now, $T'$ is $S$-centric and fully normalized in $\F$, and so $T'$ is $\F$-centric. However, this implies that $T$ is $\F$-centric as well, a contradiction. 
\end{proof}

\section{Proof of Theorems \ref{th:main} and \ref{main}}
We now focus on a Sylow $p$-subgroup of ${\rm G}_2(p)$, for $p\geq 5$, obtaining the final ingredients to prove Theorems \ref{th:main} and \ref{main}.

\begin{proof}[Proof of Theorem \ref{main}]
Aiming for a contradiction, suppose Theorem \ref{main} is false. Hence there are subgroups $Q,P,R \leq S$ and a  simple $\mathbb{F}_p \Out_\F(Q)$-module $V$ such that $P$ and $R$ are $\F$-centric, $Q$ and $T:=P\cap R$ are not $\F$-centric and the composition \begin{equation*}
S_{Q,V}(P) \xrightarrow{\Res_T^P} S_{Q, V}(T) \xrightarrow{\Ind_T^R} S_{Q, V}(R)	
\end{equation*}
is non-zero. Hence Hypothesis \ref{hp:contra.max} holds.

We first show that $ \N_S(T) = \gamma_1(S)$. By Lemma \ref{lem:T.not.S.centric} either $\N_S(T) = \gamma_1(S)$ or $T\norm S$. Aiming for a contradiction, suppose the latter holds.
By Lemma \ref{lem:Tnorm} we get $T=\gamma_i(S)$ for some $i\geq 2$ such that $\gamma_{i-1}(S)$ is an abelian subgroup of $\gamma_1(S)$. Since $|S|=p^6$ and $\gamma_1(S)$ is extraspecial, we deduce that  either $T=Z(S)$ or $T=Z_2(S)$, contradicting Lemmas \ref{lem:T.not.center} and \ref{lem:Tnorm}. Hence $T$ is not normal in $S$ and we conclude that $\N_S(T) = \gamma_1(S)$.

In particular $|T|\leq p^4$. Also, by Lemma \ref{lem:T.not.center} the group $T$ has order at least $p^2$. If $|T|=p^4$ then $P$ and $R$ are maximal subgroups of $S$. Since $S$ has maximal nilpotency class, the intersection of any pair of distinct maximal subgroups of $S$ coincides with the group $\gamma_2(S)$. Hence $T =\gamma_2(S)$, contradicting the fact that $T$ is not normal in $S$. Thus $p^2 \leq |T| \leq p^3$. 
By Lemma \ref{lem:T.not.S.centric}, the group $T$ is not $S$-centric, so if $|T|=p^3$ then it is not abelian.

Let's now consider the groups $P$ and $R$.  Recall that $P$ and $R$ are distinct $\F$-centric subgroups of $S$ that are non abelian by Lemma \ref{lem:basic.prop} (iv).
By Lemma \ref{lem:T.contains.centralizers} we have $Z(P)\leq T$ and $Z(R)\leq T$. Since $Z_2(S)\neq T$ and if $|T|=p^3$ then $T$ is not abelian, we deduce that $P,R \nleq C_S(Z_2(S))$.  

Let $U\in \{P,R\}$.
\begin{description}

\item[Claim 1] If $U\nleq \gamma_1(S)$ then $\gamma_3(S)\leq U \cap \gamma_1(S) < U$ and $|U|\geq p^4$. 

\emph{Proof:} Set $U_0 = U \cap \gamma_1(S)$. By \cite[Lemma 2.12]{pearls} we have $U_0 = \gamma_i(S)$ for some $i\geq 2$ and $[U \colon U_0]=p$. Since $|U|\geq p^3$, we deduce that $Z_2(S) = \gamma_4(S) \leq U_0$, and so $TZ_2(S) \leq U_0$. This proves $|U_0|\geq p^3$, and so $\gamma_3(S)\leq U_0 < U$ and $|U|\geq p^4$.

\item[Claim 2] If $U\leq \gamma_1(S)$ then either $U=\gamma_1(S)$ or $|U|=p^4$ and $T = Z(U)$ has order $p^2$.

\emph{Proof:} Since $U$ is $S$-centric nonabelian, we deduce that either $U=\gamma_1(S)$ or $|U|=p^4$ and $|Z(U)| =p^2$. In the latter case,  since $Z(U) \leq T$ and $|T|\leq p^3$, we deduce that $T$ is abelian and so $|T|=p^2$, that is, $T=Z(U)$.

\item[Claim 3]  If $U\leq \gamma_1(S)$ then $U=\gamma_1(S)$.

\emph{Proof:} Aiming for a contradiction, suppose $U<\gamma_1(S)$. Hence by Claim 2 $|U|=p^4$ and $T=Z(U)$. In particular, $T \norm U$ and $T < C_U(T)=U$, contradicting Lemma \ref{lem:T.abelian.normal}(i).

\end{description} 

Suppose first that both $P$ and $R$ are not contained in $\gamma_1(S)$. Then Claim 1 implies that $\gamma_3(S) \leq P \cap R =T$ and so $T=\gamma_3(S)$, contradicting the fact that $T$ is not normal in $S$. Thus, without loss of generality, we can suppose $P\leq \gamma_1(S)$. Hence by Claim 3 we have $P=\gamma_1(S)$ and  $R \nleq \gamma_1(S)$. Again by Claim 1, $\gamma_3(S) < R$. therefore $\gamma_3(S) \leq P\cap R = T$, a contradiction. This completes the proof.
\end{proof}

\begin{proof}[Proof of Theorem \ref{th:main}]
By Theorem \ref{main}, for any choice of the subgroups $Q, P$ and $R$ as in the statement of Theorem \ref{main}, the composition (\ref{composition}) is zero. Now Proposition \ref{prop:trunc} guarantees that the restriction $S_{Q,V}|_{\orb(\F^c)}$ is an $\F^c$-truncated Mackey functor for $\F$. Hence \cite[Theorem A]{DP} implies that 
\[{\lim\limits_{\longleftarrow}}_{\orb(\F^c)}^i {S_{Q,V}}^*|_{\orb(\F^c)} = 0\] for all $i > 0$. We conclude by \cite[Proposition 4.3]{DP}. 
\end{proof}

\bibliographystyle{alpha}
\bibliography{mybooks}

\end{document}